\documentclass[american]{article}
\usepackage[T1]{fontenc}
\usepackage[utf8]{inputenc}
\usepackage{mathrsfs}
\usepackage{amsmath}
\usepackage{amsthm}
\usepackage{amssymb}

\makeatletter

\newcommand{\noun}[1]{\textsc{#1}}

\usepackage{tikz}
\usetikzlibrary{positioning, shapes.geometric, fit}
\makeatother

\theoremstyle{plain}
\newtheorem{thm}{\protect\theoremname}
\newtheorem{lem}[thm]{\protect\lemmaname}
\theoremstyle{remark}
\newtheorem{claim}[thm]{\protect\claimname}
\theoremstyle{definition}
\newtheorem{example}[thm]{\protect\examplename}
\theoremstyle{remark}
\newtheorem*{rem*}{\protect\remarkname}
\usepackage{babel}
\providecommand{\claimname}{Claim}
\providecommand{\examplename}{Example}
\providecommand{\lemmaname}{Lemma}
\providecommand{\remarkname}{Remark}
\providecommand{\theoremname}{Theorem}

\begin{document}
\global\long\def\frak#1{\mathfrak{#1}}%
\global\long\def\cal#1{\mathcal{#1}}%
\global\long\def\surr#1{{#1}}%
\global\long\def\t#1{\mathrm{#1}}%
\global\long\def\of#1{{\left(#1\right)}}%
\global\long\def\size#1{\left\Vert #1\right\Vert }%
\global\long\def\bof#1{{\left[#1\right]}}%
\global\long\def\lsup#1#2{{}{}^{#1}{#2}}%
\global\long\def\res{{\restriction\,}}%

\global\long\def\super{\supseteq}%
\global\long\def\subnq{\subsetneqq}%
\global\long\def\all{\forall}%
\global\long\def\exi{\exists}%
\global\long\def\sub{\subseteq}%
\global\long\def\empa{\varnothing}%
\global\long\def\sem{\smallsetminus}%
\global\long\def\ity{\infty}%
\global\long\def\emp{\varnothing}%

\global\long\def\bcup{\bigcup}%
\global\long\def\bcap{\bigcap}%
\global\long\def\and{\wedge}%
\global\long\def\orr{\vee}%
\global\long\def\then{\Rightarrow}%
\global\long\def\tm{\times}%
\global\long\def\isom{\cong}%
\global\long\def\ds{\dots}%
\global\long\def\conc{^{\frown}}%
\global\long\def\nin{\notin}%
\global\long\def\mto{\mapsto}%
\global\long\def\toby#1{\stackrel{#1}{\longrightarrow}}%
\global\long\def\cont{\frak c}%
\global\long\def\sym{\bigtriangleup}%
\global\long\def\clb#1{\overline{#1}}%
\global\long\def\cl{\mathrm{cl}}%
\global\long\def\clof#1{\mathrm{cl}\left(#1\right)}%

\global\long\def\brp{\mathbb{R}^{+}}%
\global\long\def\br{\mathbb{R}}%
\global\long\def\bz{\mathbb{Z}}%
\global\long\def\bq{\mathbb{Q}}%
\global\long\def\bc{\mathbb{C}}%
\global\long\def\bn{\mathbb{N}}%
\global\long\def\bg{\mathbb{G}}%
\global\long\def\ba{\mathbb{A}}%
\global\long\def\bd{\mathbb{D}}%
\global\long\def\bp{\mathbb{P}}%
\global\long\def\bu{\mathbb{U}}%
\global\long\def\bl{\mathbb{L}}%
\global\long\def\bs{\mathbb{S}}%
\global\long\def\bk{\mathbb{K}}%

\global\long\def\gw{\omega}%
\global\long\def\ga{\alpha}%
\global\long\def\gb{\beta}%
\global\long\def\gd{\delta}%
\global\long\def\ph{\varphi}%
\global\long\def\gga{\gamma}%
\global\long\def\gt{\theta}%
\global\long\def\gz{\zeta}%
\global\long\def\gk{\kappa}%
\global\long\def\gs{\sigma}%
\global\long\def\gl{\lambda}%
\global\long\def\gx{\xi}%
\global\long\def\gp{\pi}%
\global\long\def\eps{\varepsilon}%
\global\long\def\th{\vartheta}%
\global\long\def\vom{\varpi}%
\global\long\def\vrh{\varrho}%
\global\long\def\brh{\bar{\varrho}}%

\global\long\def\vxi{\varXi}%
\global\long\def\del{\Delta}%
\global\long\def\gam{\Gamma}%
\global\long\def\al{\aleph}%

\global\long\def\Ord{\text{Ord}}%
\global\long\def\dom{\operatorname{dom}}%
\global\long\def\cof{\operatorname{cf}}%
\global\long\def\ran{\operatorname{ran}}%
\global\long\def\im{\operatorname{im}}%
\global\long\def\ord{\text{On}}%

\global\long\def\abs{\left|\,\right|}%
\global\long\def\mr{\diagdown}%
\global\long\def\mc{\diagup}%
\global\long\def\sl{\cal L}%
\global\long\def\sc{\cal C}%

\global\long\def\Bs{\mathbf{\Sigma}}%
\global\long\def\Bp{\mathbf{\Pi}}%

\global\long\def\cc{\Gamma}%
\global\long\def\dd{\{\}}%
\global\long\def\rs#1{{\restriction_{#1}}}%

\global\long\def\abv#1{\left|#1\right|}%
\global\long\def\set#1{\left\{  #1\right\}  }%
\global\long\def\cur#1{\left(#1\right)}%
\global\long\def\cei#1{\left\lceil #1\right\rceil }%
\global\long\def\ang#1{\left\langle #1\right\rangle }%
 
\global\long\def\bra#1{\left[#1\right]}%

\global\long\def\se{\mathscr{E}}%
\global\long\def\sc{\mathscr{C}}%
\global\long\def\sp{\mathscr{P}}%
\global\long\def\sb{\mathscr{B}}%
\global\long\def\sa{\mathscr{A}}%
\global\long\def\ss{\mathscr{S}}%
\global\long\def\sf{\mathscr{F}}%
\global\long\def\sx{\mathscr{X}}%
\global\long\def\sh{\mathscr{H}}%
\global\long\def\su{\mathscr{U}}%
\global\long\def\so{\mathscr{O}}%
\global\long\def\sy{\mathscr{Y}}%
\global\long\def\sm{\mathscr{M}}%
\global\long\def\sn{\mathscr{N}}%
\global\long\def\sd{\mathscr{D}}%
\global\long\def\sl{\mathscr{L}}%
\global\long\def\sg{\mathscr{G}}%
\global\long\def\sk{\mathscr{K}}%
\global\long\def\sv{\mathscr{V}}%
\global\long\def\sw{\mathscr{W}}%
\global\long\def\st{\mathscr{T}}%
\global\long\def\si{\mathscr{I}}%
\global\long\def\sj{\mathscr{J}}%
\global\long\def\sr{\mathscr{R}}%

\global\long\def\df#1{\boldsymbol{#1}}%
\global\long\def\der#1{\overset{\centerdot}{#1}}%
\global\long\def\ocl#1{\left(#1\right]}%
\global\long\def\clo#1{\left[#1\right)}%

\global\long\def\qand{\qquad\text{and}\qquad}%
\global\long\def\ovl#1{\overline{#1}}%
\global\long\def\ome{\Omega}%

\global\long\def\dof#1{\text{dist}\left(#1\right)}%
\global\long\def\pt{\partial}%
\global\long\def\nm#1{\left\Vert #1\right\Vert }%
\global\long\def\nmd{\left\Vert \cdot\right\Vert }%

\global\long\def\ceic#1{\cei{#1}^{\circ}}%
\global\long\def\fork{\varUpsilon}%
\global\long\def\ncup{\biguplus}%
\global\long\def\md{\t -}%
\global\long\def\gka{\varkappa}%

\global\long\def\col#1{\operatorname{col}\left(#1\right)}%

\global\long\def\hei#1{\operatorname{height}\left(#1\right)}%

\global\long\def\tr#1{\operatorname{Tr}\left(#1\right)}%

\global\long\def\fin#1{\operatorname{Fin}\left(#1\right)}%

\global\long\def\abp#1#2{#1\text{-}#2}%

\global\long\def\bt{\Phi}%
\global\long\def\btc{\bt^{\circ}}%
\global\long\def\vt{\varsigma}%
\global\long\def\up{{\uparrow}}%
\global\long\def\down{{\downarrow}}%

\title{Augmentation Lemma for Halin’s Conjecture}
\author{\noun{Jerzy Wojciechowski}\\
West Virginia University\\
Morgantown, WV, U.S.A.}
\maketitle
\begin{abstract}
The longstanding conjecture of Halin characterizing the existence
of normal spanning trees in infinite graphs has been recently proved
by Max Pitz \cite{Pitz}. A critical step in the proof involves the
construction of dominated torsos, whose properties are essential to
the overall proof. In this note, we provide a correction to the proof
of a key property of this construction.
\end{abstract}

\section{Introduction.}

We follow the notation in \cite{Diestel}. A rooted spanning tree
$T$ of a graph $G$ is \emph{normal} if the ends of any edge of $G$
are comparable in the natural tree order of $T$. A graph $G$ has
\emph{countable coloring number} if there is a well-order $\le^{*}$
on $V\of G$ such that every vertex of $G$ has only finitely many
neighbors preceding it in $\le^{*}$. Halin \cite{Halin} conjectured
and Pitz \cite{Pitz} proved the following result.
\begin{thm}
\label{thm1}A connected graph has a normal spanning tree if and only
if every minor of it has countable coloring number.
\end{thm}

A set $U\sub V\of G$ of vertices of a graph $G$ is \emph{dispersed}
in $G$ if every ray in $G$ can be separated from $U$ be a finite
set of vertices. Given a subgraph $H$ of $G$, a set of vertices
$A\sub V\of H$ of the form $A=N_{G}\of D$ for some component $D$
of $G-H$ is an \emph{adhesion set}. A subgraph $H\sub G$ has \emph{finite
adhesion} if all adhesion sets in $H$ are finite.

The following Decomposition Lemma (Lemma 3.3 in \cite{Pitz}) provides
the backbone of the proof of Theorem \ref{thm1}\@.
\begin{lem}
[Decomposition Lemma]\label{lem98}Let $G$ be a connected graph
of uncountable size $\gk$ with the property that all its minors have
countable coloring number. Then $G$ can be written as a continuous
increasing union $\bigcup_{i\in\gs}G_{i}$ of infinite, $<\gk$-sized
connected induced subgraphs $G_{i}$ of finite adhesion in $G$.
\end{lem}

To use Lemma \ref{lem98} in the proof of Theorem \ref{thm1}, the
key property is the following claim (Claim 4.1 in \cite{Pitz}) about
the graphs $G_{i}$ from Lemma \ref{lem98}.
\begin{claim}
\label{claim3}Every $V\of{G_{i}}$ is a countable union of dispersed
sets in $G$.
\end{claim}

Let $H$ be a subgraph of a graph $G$. An \emph{augmentation} of
$H$ in $G$ is a contraction minor $K$ of $G$ that contains $H$
as a subgraph. An augmentation $K$ is \emph{faithful} if every subset
$U\sub V_{H}$ that is dispersed in $K$ is also dispersed in $G$.
It is \emph{conservative} if $\abv{V_{K}}=\abv{V_{H}}$.

The proof of Claim \ref{claim3} in \cite{Pitz} involves the construction
of a \emph{dominated torso} $\hat{G_{i}}$ of $G_{i}$ that is a connected,
faithful and conservative augmentation of $G_{i}$ in $G$. Using
the general notation, the dominated torso $K$ of $H$ in $G$ is
defined as follows.

Let $\sa$ be the set of all adhesion sets of $H$ in $G$. Thus $A\in\sa$
if and only if $A=N_{G}\of D$ for some component $D$ of $G-H$.
For each $A\in\sa$ let $\sd_{A}$ be the set of components $D$ of
$G-H$ with $N_{G}\of D=A$. Let $\sa'$ be the set of all $A\in\sa$
such that $\sd_{A}$ is finite and let $\sa'':=\sa\sem\sa'$. Let
\[
\sd':=\bigcup_{A\in\sa'}\sd_{A}\qand\sd'':=\bigcup_{A\in\sa''}\sd_{A}.
\]

Let $\eta:\sd''\to V\of H$ be a function such that for every $A\in\sa''$
the restriction $\eta_{A}$ of $\eta$ to $\sd_{A}$ is a surjection
onto $A$. Since each $A\in\sa''$ is finite and the corresponding
$\sd_{A}$ is infinite, such a surjection $\eta_{A}$ exists.

Let $\set{v_{D}:D\in\sd'}$ be a set disjoint with $V\of G$ such
that $v_{D}\neq v_{D'}$ for $D\neq D'$, and let
\[
V\of K:=V\of H\cup\set{v_{D}:D\in\sd'}.
\]
Define $\vrh:V\of G\to V\of K$ by:
\[
\vrh\of u:=\begin{cases}
u & \text{if }u\in V\of H;\\
v_{D} & \text{if }u\in V\of D\text{ for some }D\in\sd';\\
\eta\of D & \text{if }u\in V\of D\text{ for some }D\in\sd''.
\end{cases}
\]
Let $K$ be the unique contraction minor of $G$ witnessed by $\vrh$.

Define a \emph{tendril} in $G$ to be a ray $S$ that has infinitely
many vertices in the subgraph $H$.

To prove that $K$ is a faithful augmentation of $H$ in $G$, given
a tendril $S$ in $G$, Pitz uses a canonical projection $S'$ of
$S$ onto $K$. We will call it a $K$-projection of $S$. It is defined
as follows.

Given a tendril $S=\cur{v_{n}}_{n\in\gw}$ in $G$, the $H$-\emph{masking}
of $S$ is the sequence $\cur{v_{n}'}_{n\in\gw}$, where each $v_{n}':=v_{n}$
if $v_{n}\in V\of H$, and $v_{n}':=D$ if $v_{n}\in V\of D$ for
some component $D$ of $G-H$.

The $K$-\emph{projection} of the tendril $S$ is the sequence $S':=\cur{w_{m}}_{m\in\gw}$
of vertices of $K$ obtained from the $H$-masking $\cur{v_{n}'}_{n\in\gw}$
of $S$ by:
\begin{enumerate}
\item Replacing each maximal subsequence of identical consecutive terms
$D\in\sd'$ with the single vertex $v_{D}\in V_{K}$.
\item Removing all terms $D\in\sd''$ from the sequence entirely.
\end{enumerate}
It is clear from the construction that $v_{D}$ is adjacent in $K$
to all vertices in $N_{G}\of D$ for every $D\in\sd'$. Moreover,
for every $A\in\sa''$ the induces subgraph $K\bof A$ of $K$ is
complete. Thus the $K$-projection $S'=\cur{w_{m}}_{m\in\gw}$ of
$S$ is a locally finite tour in $K$ (infinite walk visiting each
vertex only finitely many times).

Let $U\sub V\of H$ be dispersed in $K$. To show that $U$ is dispersed
in $G$, we need to show, in particular, that if $S$ is a tendril
in $G$, then $U$ is finitely separated from $S$.

Let $S'$ be the $K$-projection of $S$. Since $U$ is dispersed
in $K$ there is a finite set $F$ separating $U$ from $S'$ in $K$.
Pitz in \cite{Pitz} claims that the set
\[
X:=\cur{F\cap V\of H}\cup\bigcup\set{N_{G}\of D:D\in\sd',\,v_{D}\in F}
\]
separates $U$ from $S$ in $G$.

Consider the following example, which demonstrates that this claim
may fail.
\begin{example}
\label{exa4}Let $G$ be a graph whose vertex set it the disjoint
union $X\cup Y\cup Z$, where $X,Y$ are countably infinite sets enumerated
as $X:=\set{x_{j}:j\in\gw}$, $Y:=\set{y_{j}:j\in\gw}$ and $Z$ is
enumerated as $Z:=\set{z_{\ga}:\ga\in\gw_{1}}$. The edge set of $G$
consists of the edges $x_{j}x_{j+1}$, $x_{j}y_{j}$, $x_{j+1}y_{j}$,
for each $j\in\gw$ and the edges $x_{1}z_{\ga}$, $x_{2}z_{\ga}$
and $x_{3}z_{\ga}$ for each $\ga\in\gw_{1}$. Let $H:=G\bof X$.
The subgraph $H$ has finite adhesion in $G$. Moreover, $\sd'=\set{D_{i}':i\in\gw}$,
where each component $D_{i}'$ has a single vertex $y_{i}$ and $\sd''=\set{D_{\ga}'':\ga\in\gw_{1}}$,
where each $D_{\ga}''$ has a single vertex $z_{\ga}$. The minor
$K$ has vertex set $X\cup\set{v_{D_{i}'}:i\in\gw}$ with edges $x_{i}x_{i+1}$,
$x_{i}v_{D_{i}'}$ and $x_{i+1}v_{D_{i}'}$ for each $i\in\gw$ and
the edge $x_{1}x_{3}$.

Let $S$ be the ray 
\[
\cur{x_{2},z_{0},x_{3},y_{3},x_{4},y_{4},x_{5},\ds,x_{n},y_{n},x_{n+1},y_{n+1},\ds}
\]
in $G$. The $K$-projection of $S$ is the sequence $S':=\cur{x_{2},x_{3},v_{D_{3}'},x_{4},v_{D_{4}'},\ds}$.
Let $U:=\set{x_{0}}$ and $F:=\set{x_{2},x_{3}}$. Then $F$ separates
$U$ from $S'$ (see Figure \ref{fig1-1}), 
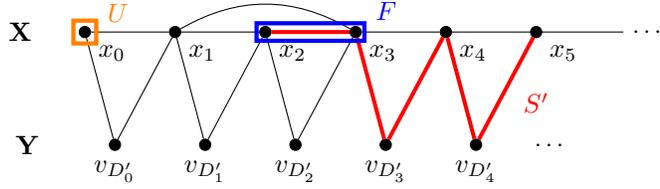
\begin{figure}
\begin{tikzpicture}[
    vertex/.style={circle, draw, fill=black, inner sep=1.5pt},
    label_node/.style={font=\small},
    group/.style={draw, dashed, rounded corners, inner sep=10pt}
]

    \node[vertex, label=below right:$x_0$] (x0) at (0,0) {};
    \node[vertex, label=below right:$x_1$] (x1) at (1.2,0) {};
    \node[vertex, label=below right:$x_2$] (x2) at (2.4,0) {};
    \node[vertex, label=below right:$x_3$] (x3) at (3.6,0) {};
    \node[vertex, label=below right:$x_4$] (x4) at (4.8,0) {};
    \node[vertex, label=below right:$x_5$] (x5) at (6,0) {};
    \node[label_node] (x_dots) at (7.5,0) {$\dots$};
    \node[left=0.5cm of x0] {\textbf{X}};

	\node[vertex, label=below:$v_{D_0'}$] (y0) at (.4,-1.5) {};
	\node[vertex, label=below:$v_{D_1'}$] (y1) at (1.6,-1.5) {};
	\node[vertex, label=below:$v_{D_2'}$] (y2) at (2.8,-1.5) {};
    \node[vertex, label=below:$v_{D_3'}$] (y3) at (4,-1.5) {};
    \node[vertex, label=below:$v_{D_4'}$] (y4) at (5.2,-1.5) {};
    \node[label_node] (y_dots) at (6.2,-1.5) {$\dots$};
    \node[left=.8cm of y0] {\textbf{Y}};

    \draw (x0) -- (x1) -- (x2) -- (x3) -- (x4) -- (x_dots);
    
	\draw (x1) to [bend left] (x3);
    \draw[line width=1.5pt, red] (x2) -- (x3);
    \draw[line width=1.5pt, red] (x3) -- (y3);
    \draw[line width=1.5pt, red] (x4) -- (y3);
    \draw[line width=1.5pt, red] (x4) -- (y4);
    \draw[line width=1.5pt, red] (x5) -- (y4);
\node[anchor=south] at (6,-1.2) {\textcolor{red}{$S'$}};

    \draw (x0) -- (y0);
    \draw (x1) -- (y0);
    \draw (x1) -- (y1);
    \draw (x2) -- (y1);
    \draw (x2) -- (y2);
    \draw (x3) -- (y2);

    \node[draw=blue, ultra thick, fit=(x2) (x3), inner sep=1pt, label=3:\textcolor{blue}{$F$}] {};
	\node[draw=orange, ultra thick, fit=(x0), inner sep=2pt, label=3:\textcolor{orange}{$U$}] {};
\end{tikzpicture}

\caption{$F$ separates $U$ from $S'$ in $K$}\label{fig1-1}
\end{figure}
 but $X=F$ does not separate $U$ from $S$ since $\cur{x_{0},x_{1},z_{0}}$
is a path joining $x_{0}\in U$ and $z_{0}\in S$ (see Figure \ref{fig1}).
\begin{figure}
\begin{centering}
\begin{tikzpicture}[
    vertex/.style={circle, draw, fill=black, inner sep=1.5pt},
    label_node/.style={font=\small},
    group/.style={draw, dashed, rounded corners, inner sep=10pt}
]

    \node[vertex, label=below right:$x_0$] (x0) at (0,0) {};
    \node[vertex, label=below right:$x_1$] (x1) at (1.2,0) {};
    \node[vertex, label=below right:$x_2$] (x2) at (2.4,0) {};
    \node[vertex, label=below right:$x_3$] (x3) at (3.6,0) {};
    \node[vertex, label=below right:$x_4$] (x4) at (4.8,0) {};
    \node[vertex, label=below right:$x_5$] (x5) at (6,0) {};
    \node[label_node] (x_dots) at (7.5,0) {$\dots$};
    \node[left=0.5cm of x0] {\textbf{X}};

	\node[vertex, label=below:$y_0$] (y0) at (.4,-1.5) {};
	\node[vertex, label=below:$y_1$] (y1) at (1.6,-1.5) {};
	\node[vertex, label=below:$y_2$] (y2) at (2.8,-1.5) {};
    \node[vertex, label=below:$y_3$] (y3) at (4,-1.5) {};
    \node[vertex, label=below:$y_4$] (y4) at (5.2,-1.5) {};
    \node[label_node] (y_dots) at (6.2,-1.5) {$\dots$};
    \node[left=.8cm of y0] {\textbf{Y}};

    \node[vertex, label=above:$z_0$] (z0) at (1.6,1.5) {};
	\node[vertex, label=above:$z_1$] (z1) at (2.8,1.5) {};
	\node[vertex, label=above:$z_2$] (z2) at (4,1.5) {};
    \node[label_node] (z_dots) at (5,1.5) {$\dots$};
    \node[left=1cm of z0] {\textbf{Z}};

    \draw (x0) -- (x1) -- (x2) -- (x3) -- (x4) -- (x_dots);
    
    \draw[line width=1.5pt, red] (x2) -- (z0);
    \draw[line width=1.5pt, red] (x3) -- (z0);
    \draw[line width=1.5pt, red] (x3) -- (y3);
    \draw[line width=1.5pt, red] (x4) -- (y3);
    \draw[line width=1.5pt, red] (x4) -- (y4);
    \draw[line width=1.5pt, red] (x5) -- (y4);
\node[anchor=south] at (6.5,-1.2) {\textcolor{red}{tendril $S$}};
    \draw (x2) -- (z0);
    \draw (x3) -- (z0);
    \draw (x1) -- (z1);
    \draw (x2) -- (z1);
    \draw (x3) -- (z1);
    \draw (x1) -- (z2);
    \draw (x2) -- (z2);
    \draw (x3) -- (z2);

    \draw (x0) -- (y0);
    \draw (x1) -- (y0);
    \draw (x1) -- (y1);
    \draw (x2) -- (y1);
    \draw (x2) -- (y2);
    \draw (x3) -- (y2);
    \draw (x3) -- (y3);
    \draw (x4) -- (y3);
    \draw (x4) -- (y4);
    \draw (x5) -- (y4);

    \node[draw=blue, ultra thick, fit=(x2) (x3), inner sep=1pt, label=3:\textcolor{blue}{$X$}] {};

    \draw[line width=1.5pt, green!60!black] (x0) -- (x1) -- (z0);
    \node[anchor=south] at (-.2,.5) {\textcolor{green!60!black}{path from $U$ to $S$}};
	\node[draw=orange, ultra thick, fit=(x0), inner sep=2pt, label=250:\textcolor{orange}{$U$}] {};

\end{tikzpicture}
\par\end{centering}
\caption{$X$ does not separated $U$ from $S$ in $G$}\label{fig1}
\end{figure}
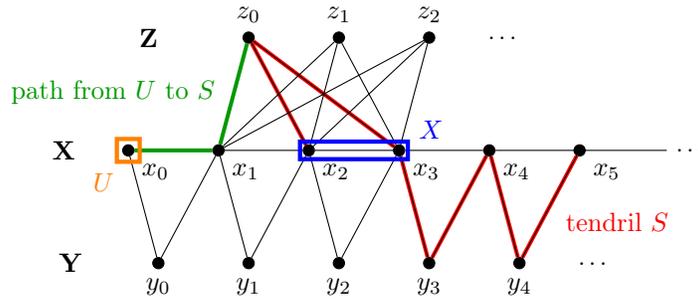
\end{example}

The construction of dominated torsos in \cite{Pitz} provides an implicit
proof of the following Augmentation Lemma. However, the part of this
proof that demonstrates faithfulness is not quite correct (see Example
\ref{exa4}). The purpose of this note is to provide a corrected proof 
\begin{lem}
[Augmentation Lemma]\label{lem99}Let $G$ be an infinite connected
graph and $H$ be an infinite connected induced subgraph with finite
adhesion in $G$. Then there exists a connected, faithful, and conservative
augmentation $K$ of $H$ in $G$.
\end{lem}

\section{Proof of the Augmentation Lemma.}

Let $G$ be an infinite connected graph and $H$ be an infinite connected
induced subgraph with finite adhesion in $G$. We define $K$ as the
dominated torso of $H$ in $G$ as described above. The minor $K$
is clearly connected and as proved in \cite{Pitz} it is conservative.
It remains to show that $K$ is a faithful augmentation.

Recall that $\sa$ is the set of all adhesion sets of $H$ in $G$,
that $\sa'$ is the set of all $A\in\sa$ such that $\sd_{A}$ is
finite and $\sa'':=\sa\sem\sa'$. Moreover,
\[
\sd'=\bigcup_{A\in\sa'}\sd_{A}\qand\sd''=\bigcup_{A\in\sa''}\sd_{A}.
\]

Suppose $F\sub V_{K}$ and $S=\cur{v_{n}}_{n\in\gw}$ is a tendril
in $G$. We define the subset $F_{S}\sub V\of H$ as follows. Let
$\cur{v_{n}'}_{n\in\gw}$ be the $H$-masking of $S$. Define the
auxiliary sets of components:
\[
\hat{\sd}':=\set{D\in\sd':v_{D}\in F}
\]
and
\[
\hat{\sd}'':=\set{D\in\sd'':v_{n}'=D\text{ and }v_{n+1}\in F\text{ for some }n\in\gw}.
\]
Note that $\hat{\sd}'$ is independent of $S$, but $\hat{\sd}''$
depends on $S$. Let $\hat{\sd}:=\hat{\sd}'\cup\hat{\sd}''$ and
\[
F_{S}:=\cur{F\cap V\of H}\cup\bigcup_{D\in\hat{\sd}}N_{G}\of D.
\]
We call $F_{S}$ the $S$-\emph{modification} of $F$. It is clear
that if $F$ is finite, then $F_{S}$ is also finite.

The difference between the $F_{S}$-modification and the set $X$
used by Pitz in his proof is the inclusion of the family $\hat{\sd}''$
in the union used to define $F_{S}$, while in the definition of $X$
only $\hat{\sd}'$ was used.
\begin{lem}
\label{lem421}Let $S$ be a tendril in $G$ with $K$-projection
$S':=\cur{w_{m}}_{m\in\gw}$. If a set $F\sub V_{K}$ separates a
subset $U\sub V_{H}$ from the set $W:=\set{w_{m}:m\in\gw}$ in $K$,
then the modification $F_{S}$ separates $U$ from $V_{S}$ in $G$.
\end{lem}

\begin{proof}
Let $P=\cur{u_{0},u_{1},\ds,u_{k}}$ be a path in $G$ joining from
vertex $u\in U$ to a vertex $v\in S$. Our goal is to show that $P$
has a vertex in $F_{S}$.

Let $\cur{u_{0}',u_{1}',\ds,u_{k}'}$ be the $H$-masking of $P$
and $\cur{x_{0},x_{1},\ds,x_{q}}$ be its $K$-projection (defined
similarly as for tendrils). Since $u_{0}\in V\of H$, we have $x_{0}=u_{0}'=u_{0}$.

First, consider the case where $u_{k}\in V_{H}$ or $u_{k}\in V_{D}$
for some $D\in\sd'$. If $u_{k}\in V_{H}$, then $x_{q}=u_{k}\in W$.
Otherwise, since $u_{k}$ is a vertex of $S$, the definition of $K$-projection
implies $x_{q}=v_{D}\in W$. Because $F$ separates $U$ from $W$
in $K$, and $\cur{x_{0},x_{1},\ds,x_{q}}$ is a walk in $K$ connecting
$u\in U$ to $x_{q}\in W$, there must be some index $i\in\set{0,1,\ds,q}$
such that $x_{i}\in F$. \smallskip

\noindent$\bigstar$ If $x_{i}\in V_{H}$, then by the definition
of the $S$-modification, $x_{i}\in F_{S}$, which satisfies the requirement.
\smallskip

\noindent$\bigstar$ If $x_{i}=v_{D}$ for some component $D\in\sd'$,
then $v_{D}\in F$ implies that $D\in\hat{\sd}'$. Since $x_{0}\in V_{H}$,
we know that $i\ge1$ and $x_{i-1}\in V_{H}$. Thus $x_{i-1}$ is
a vertex of $P$ that lies in $N_{G}\of D$. Since $D\in\hat{\sd}'$,
the adhesion set $N_{G}\of D$ is contained in $F_{S}$. Therefore,
$x_{i-1}\in F_{S}$, as required. \smallskip

Now assume that $u_{k}\in V\of D$ for some $D\in\sd''$. Then $A:=N_{G}\of D\in\sa''$.
Since $u_{0}\in V\of H$, there is the largest index $j\in\set{0,1,\ds,k}$
with $u_{j}\in V\of H$. Then $u_{j+1}'=u_{j+2}'=\ds=u_{k}'=D$ and
$x_{q}=u_{j}\in A$. Let $S=\cur{v_{n}}_{n\in\gw}$ and $\cur{v_{n}'}_{n\in\gw}$
be the $H$-masking of $S$. Since $u_{k}=v_{n}$ for some $n\in\gw$,
it follows that $v_{n}'=D$. Since $S$ is a tendril, there is a smallest
integer $s>n$ with $v_{s}\in V\of H$. Then $v_{s-1}\in V\of D$
so $v_{s}\in A$. 

Note the graph $K\bof A$ is complete. Since both $x_{q}$ and $v_{s}$
belong to $A$ and $A\in\sa''$, it follows that $x_{q}v_{s}\in E\of K$.
Then $\cur{x_{0},x_{1},\ds,x_{q},v_{s}}$ is a walk in $K$ connecting
$u\in U$ to $v_{s}\in W$.

Because $F$ separates $U$ from $W$ in $K$, if $v_{s}\nin F$,
then there is $i\in\set{0,1,\ds,q}$ such that $x_{i}\in F$. Then,
arguing as above, we find a vertex of the path $P$ that belongs to
$F_{S}$. Suppose $v_{s}\in F$. Then the definition of the auxiliary
set $\hat{\sd}''$ implies that $D\in\hat{\sd}''$ and consequently
$x_{q}\in F_{S}$.

In either case we find a vertex of $P$ that belongs to $F_{S}$ which
implies that $F_{S}$ separates $U$ from $S$ in $G$.
\end{proof}

\medskip

\noindent\textbf{Proof that $K$ is a faithful augmentation.} Assume
that $U\sub V\of H$ is dispersed in $K$. Let $S=\cur{v_{i}}_{i\in\gw}$
be a ray in $G$. To show that $U$ is dispersed in $G$, we must
find a finite subset of $V\of G$ that separates $U$ from $S$.

First, suppose there exists $n\in\gw$ such that $v_{i}\nin V\of H$
for all $i\ge n$. Then the tail $\set{v_{i}:i\ge n}$ of $S$ is
contained in a single component $D$ of $G-H$. The union of the finite
adhesion set $N_{G}\of D$ and the finite initial segment $\set{v_{i}:i<n}$
of $S$ separates $U$ from $S$.

If no such $n$ exists, then $S$ is a tendril in $G$. Let $\cur{w_{m}}_{m\in\gw}$
be the $K$-projection of $S$. Since $U$ is dispersed in $K$, there
is a finite set $F\sub V\of G$ that separates $U$ from $\set{w_{m}:m\in\gw}$
in $K$.

The $S$-modification $F_{S}$ is a finite subset of $V\of G$ that
separates $U$ from $S$ in $G$ by Lemma \ref{lem421}.
\begin{rem*}
Max Pitz has pointed out in personal conversation that his separator
\[
X:=\cur{F\cap V\of H}\cup\bigcup_{D\in\hat{\sd}'}N_{G}\of D
\]
defined as in \cite{Pitz}, separates the tail of $S$ after $X$
from $U$. That is good enough to yield Claim \ref{claim3} (Claim
4.2 in \cite{Pitz}).
\end{rem*}

\end{document}